\documentclass[a4paper]{article}

\usepackage{mathtools, amssymb}
    \def\ge{\geqslant}
    \def\le{\leqslant}
    \def\NN{\mathbb N}
    \def\MM{\mathbb M}
    \def\FF{\mathbb F}
    \def\ii{{\ddot\imath}}
    \def\jj{{\ddot\jmath}}

    \def\GF(#1){\mathbb F_{#1}}
    \def\tr{\operatorname{tr}}
    \def\lcm{\operatorname{lcm}}
    \def\norm{\operatorname{norm}}
    \def\minpoly{\operatorname{minpoly}}
    \def\charpoly{\operatorname{charpoly}}

\usepackage{amsthm}
    \newtheorem{theorem}{Theorem}
    \newtheorem{fact}[theorem]{Fact} 
    
    \newtheorem{claim}[theorem]{Claim}
    \newtheorem{corollary}[theorem]{Corollary}
     
    \theoremstyle{definition}
    
    \theoremstyle{remark}

\usepackage{tikz-cd}
    \def\sma#1{\vscale{\mat{#1}}}
    \def\ssm#1{\vscale{\sma{#1}}}
    \def\bma#1{\left[\mat{#1}\right]}
    
    \def\bsm#1{\left[\sma{#1}\right]}
    \def\psm#1{\left(\sma{#1}\right)}
    \def\mat#1{\begin{matrix}#1\end{matrix}}
    \def\vscale#1{\vcenter{\hbox{\scalebox{0.8}{$#1$}}}}
    \def\hline{\noalign{\vskip.8pt\hrule height.4pt\vskip.8pt}}
    \def\lsm#1{\delimitershortfall3pt\left\lfloor\overline{\sma{#1}}\right|}
    \def\lss#1{\delimitershortfall3pt\left\lfloor\overline{\ssm{#1}}\right|}

\usepackage{xurl, hyperref}
    \hypersetup{colorlinks, allcolors=blue!50!purple}

\begin{document}

               \title{Matrix Representations of Finite Fields}
               \author{Tzu-Wei Lin, Bo-Jiun Lee, Hsin-Po Wang}
                                 \maketitle 

\begin{abstract} 
    Finite fields are important algebraic structures that have a wide range
    of applications in fields such as coding theory and cryptography.  But
    the standard construction of finite field extensions through polynomial
    quotients is computationally opaque, especially when we want to identify
    a degree-$2$ extension of $\GF(8)$ and a degree-$3$ extension of
    $\GF(4)$.

    In this short note,
    we present a coherent family of representations by matrices
    $\rho_q^n\colon \GF(q^n) \to \GF(q)^{n\times n}$
    for all prime powers $q$ and all degrees $n \ge 1$.
    These maps are chosen so that concatenating $\rho_{q^n}^m$
    and $\rho_q^n$ recovers $\rho_q^{nm}$
    up to row and column permutations.
    As a consequence, the images of $\rho_2^6$ can be partitioned
    into four $3 \times 3$ blocks or nine $2 \times 2$ blocks
    to visualize the subfield chains $\GF(64) / \GF(8) / \GF(2)$
    and $\GF(64) / \GF(4) / \GF(2)$ at the same time.
    A variant $\varrho$ is also discussed, wherein the Frobenius automorphism
    is represented by a cyclic shift of rows and columns.

    From an educational point of view, these rhos give
    explicit and self-contained mental models of finite fields;
    subfields, trace, norm, minimal polynomial, and Frobenius all become
    visible through matrix algebra accessible to most students.
    From a theoretical point of view, the construction exhibits structural
    implications of Conway polynomials and the normal basis theorem.
\end{abstract}

\section{Introduction}

    Finite fields have many practical applications
    in coding theory, cryptography, randomized algorithms,
    combinatorial design, and other related fields
    \cite{MuM07,LNC09,MMP10,MuP13}.
    They are particularly useful because their elements
    can be represented by a finite amount of memory,
    and they have all the algebraic operations we like, especially division.

    There are, however, some subtleties when it comes to
    implementing them on a computer.
    Take $\GF(64)$ as an example.
    The standard approach is to find a degree-$6$ irreducible polynomial
    $f_2^6(x) = e_0 + e_1 x + \dotsb + e_6 x^6 \in \GF(2)[x]$ to form
    $\GF(64) \coloneqq \GF(2)[\epsilon] / \langle f_2^6(\epsilon) \rangle$.
    Each element of $\GF(64)$ is then encoded by $6$ bits,
    understood as the coefficients of a polynomial.
    The problem with this implementation is that it is not possible
    to compute the product of ``$111000$'' and ``$010101$''
    without looking up $f_2^6$ or the multiplication table.
    Another common implementation of $\GF(64)$
    is to represent each element as a power of $\epsilon$.
    This way, multiplication becomes addition of the exponents,
    but addition requires table lookups again.
    To sum up, standard implementations of finite fields tend to favor
    either the additive structure or the multiplicative structure,
    leaving the other opaque and dependent on a nontrivial lookup.

    Fortunately, a folklore trick makes both structures transparent at once:
    When treating $\GF(64)$ as a 6D vector space
    $\GF(2) \oplus \GF(2)\epsilon \oplus \dotsb \oplus \GF(2)\epsilon^5$
    over $\GF(2)$, each element of $\GF(64)$ can be thought of as
    a $6 \times 6$ matrix over $\GF(2)$ that encodes how it transforms
    the basis vectors by multiplication in $\GF(64)$.
    For instance, the matrix representation of $\epsilon$ is
    \begin{equation}
        \bsm{
            0 & 0 & 0 & 0 & 0 & -e_0 \\
            1 & 0 & 0 & 0 & 0 & -e_1 \\
            0 & 1 & 0 & 0 & 0 & -e_2 \\
            0 & 0 & 1 & 0 & 0 & -e_3 \\
            0 & 0 & 0 & 1 & 0 & -e_4 \\
            0 & 0 & 0 & 0 & 1 & -e_5
        } \in \GF(2)^{6\times6}                             \label{companion}
    \end{equation}
    because $\epsilon$ sends $1$, $\epsilon$, $\epsilon^2, \epsilon^3$,
    $\epsilon^4$, $\epsilon^5$ to $\epsilon$, $\epsilon^2, \epsilon^3$,
    $\epsilon^4$, $\epsilon^5$, $\epsilon^6 - f_2^6(\epsilon)
    = -e_0 - e_1 \epsilon - \dotsb - e_5 \epsilon^5$, respectively.
    This way, the addition and multiplication of $\GF(64)$
    are just the addition and multiplication of these $6 \times 6$ matrices.
    Education-wise, \eqref{companion} presents finite fields
    to students without training in abstract algebra.
    This makes topics like Reed--Solomon codes \cite{ReS60} and
    secret sharing \cite{Sha179} easier and faster to teach.
    It is particularly useful when teaching RAID \cite{Pla97},
    QR codes \cite{ISO24}, and the AES block cipher \cite{DaR02}
    because only one fixed finite field $\GF(256)$ is used.

    For more advanced topics, such as BCH codes \cite{Hoc59,BoC60},
    rank-metric codes \cite{Gab85}, and pairing-based cryptography
    \cite{GPS08}, we often need to work with a pair of
    fields---a base and its extension---at the same time.
    It is therefore desirable to have the fields presented
    in a way that visualizes the extension structure.
    To be more precise, we often want to start from $\GF(8)$ as a base field
    $\rho_2^3 \colon \GF(8) \to \GF(2)^{3\times3}$
    and build up $\GF(64)$ as a degree-$2$ extension
    $\rho_8^2 \colon \GF(64) \to \GF(8)^{2\times2}$.
    Note that we can combine these two maps to get
    \[
        \def\r{\rho_2^3}
        \GF(64) \xrightarrow{\rho_8^2} 
        \GF(8)^{2\times2} \xrightarrow{\bsm{\r&\r\\\r&\r}}
        (\GF(2)^{3\times3})^{2\times2} \cong \GF(2)^{6\times6},
    \]
    which gives us a self-contained representation of $\GF(64)$
    as an extension of $\GF(2)$.
    This is not the only path to obtain $\GF(64)$ over $\GF(2)$.
    We can also go through
    \[
        \def\r{\rho_2^2}
        \GF(64) \xrightarrow{\rho_4^3} 
        \GF(4)^{3\times3} \xrightarrow{\bsm{\r&\r&\r\\\r&\r&\r\\\r&\r&\r}}
        (\GF(2)^{2\times2})^{3\times3} \cong \GF(2)^{6\times6}.
    \]
    A priori, these two paths may produce
    different matrices in $\GF(2)^{6\times6}$ even though
    the goal is to construct the same field $\GF(64)$.

    In this short note, we argue that it is possible to present
    all finite fields in a coherent way so that any composition of
    extensions leads to the same presentation so the latter
    encodes all subfield information simultaneously.

    Here is a concrete example demonstrating what exactly we are looking for.
    First, note that the following matrix algebra is isomorphic to $\GF(4)$:
    \begin{equation}
        \left\{
            \bsm{0 & 0 \\ 0 & 0},
            \bsm{0 & 1 \\ 1 & 1},
            \bsm{1 & 1 \\ 1 & 0},
            \bsm{1 & 0 \\ 0 & 1}
        \right\}
        \subset \GF(2)^{2\times 2}                                \label{GF4}
    \end{equation}
    We name the elements $0, A^1, A^2, A^3$ and observe\footnote{
    So the superscripts are not just labels but actual exponents.}
    that $A^i \cdot A^j = A^{(i+j)\%3}$.
    Note also that the following matrix algebra is isomorphic to $\GF(8)$:
    \begin{equation}
        \left\{
            \arraycolsep2.4pt
            \!\bsm{0 & 0 & 0 \\ 0 & 0 & 0 \\ 0 & 0 & 0}\!,
            \!\bsm{0 & 0 & 1 \\ 1 & 0 & 1 \\ 0 & 1 & 0}\!,
            \!\bsm{0 & 1 & 0 \\ 0 & 1 & 1 \\ 1 & 0 & 1}\!,
            \!\bsm{1 & 0 & 1 \\ 1 & 1 & 1 \\ 0 & 1 & 1}\!,
            \!\bsm{0 & 1 & 1 \\ 1 & 1 & 0 \\ 1 & 1 & 1}\!,
            \!\bsm{1 & 1 & 1 \\ 1 & 0 & 0 \\ 1 & 1 & 0}\!,
            \!\bsm{1 & 1 & 0 \\ 0 & 0 & 1 \\ 1 & 0 & 0}\!,
            \!\bsm{1 & 0 & 0 \\ 0 & 1 & 0 \\ 0 & 0 & 1}\!
        \right\}
        \subset \GF(2)^{3\times3}                                 \label{GF8}
    \end{equation}
    We name the elements $0, B^1, \dotsc, B^7$ and observe\footnote{
    So the superscripts are not just labels but actual exponents.}
    that $B^i \cdot B^j = B^{(i+j)\%7}$.

    Now, alongside \eqref{companion}, we claim that $\GF(64)$
    is generated by the following $6 \times 6$ matrix.
    \begin{equation}
        \vscale{\left[
            \begin{array}{cc|cc|cc}
                1 & 0 & 1 & 0 & 1 & 1 \\
                0 & 1 & 0 & 1 & 1 & 0 \\\hline
                1 & 1 & 0 & 0 & 0 & 1 \\
                1 & 0 & 0 & 0 & 1 & 1 \\\hline
                1 & 0 & 1 & 1 & 0 & 0 \\
                0 & 1 & 1 & 0 & 0 & 0 \\
            \end{array}
        \right]}
        =
        \bma{
            A^3 & A^3 & A^2 \\
            A^2 &   0 & A^1 \\
            A^3 & A^2 &   0
        }
        \in \eqref{GF4}^{3\times3} \subset \GF(2)^{6\times6}      \label{2^3}
    \end{equation}
    This $6 \times 6$ binary matrix is divided into nine $2 \times 2$ blocks,
    each of which is an element of \eqref{GF4}.
    In other words, both the elements of $\{0, 1\}$ and
    the elements of \eqref{GF4} can be used to describe $\GF(64)$.
    Moreover,
    \begin{equation}
         \eqref{2^3}^{21i} = \bma{A^i \\& A^i \\&& A^i}
        \in \eqref{GF4}^{3\times3} \subset \GF(2)^{6\times6},     \label{I_3}
    \end{equation}
    meaning that $\eqref{2^3}^{21}$ does not invent a new model
    for its subfield $\GF(4)$---\eqref{GF4} embeds into
    the model of $\GF(64)$ by the most boring diagonal map.

    Permuting the rows and columns of \eqref{2^3}
    using $\psm{1 & 2 & 3 & 4 & 5 & 6 \\ 1 & 4 & 2 & 5 & 3 & 6}$, we get
    \begin{equation}
        \vscale{\left[
            \begin{array}{c|c}
                \sma{1 & 1 & 1 \\ 1 & 0 & 0 \\ 1 & 1 & 0} &
                \sma{0 & 0 & 1 \\ 1 & 0 & 1 \\ 0 & 1 & 0} \\\hline
                \sma{0 & 0 & 1 \\ 1 & 0 & 1 \\ 0 & 1 & 0} &
                \sma{1 & 1 & 0 \\ 0 & 0 & 1 \\ 1 & 0 & 0}
            \end{array}
        \right]}
        =
        \bma{
            B^5 & B^1 \\
            B^1 & B^6
        }
        \in \eqref{GF8}^{2\times2} \subset \GF(2)^{6\times6}.     \label{3^2}
    \end{equation}
    This permuted matrix is divided into four $3 \times 3$ blocks,
    each of which is an element of \eqref{GF8}.
    We claim that it also generates $\GF(64)$.
    Moreover,
    \begin{equation}
        \eqref{3^2}^{9j} = \bma{B^j \\& B^j}
        \in \eqref{GF8}^{2\times2} \subset \GF(2)^{6\times6},     \label{I_2}
    \end{equation}
    meaning that $\eqref{3^2}^9$ recovers its subfield
    $\GF(8)$ by simply repeating \eqref{GF8} two times.
    Moreover, $\eqref{2^3}^{63} = \eqref{3^2}^{63} = I_6$,
    the $6 \times 6$ identity matrix.
    $I_6$ together with $0 \cdot I_6$ recover
    $\GF(2)$ by repeating $\{0, 1\}$ six times.

    The paragraphs above suggest that \eqref{2^3} and \eqref{3^2}
    provide visualizations of the subfield chains
    $\GF(64) / \GF(8) / \GF(2)$ and $\GF(64) / \GF(4) / \GF(2)$.
    So together they provide a unified model of the subfield lattice
    \begin{equation}
        \begin{tikzpicture} [y=1cm, x=2cm]
            \tikzset{baseline={([yshift=-.5ex]current bounding box.center)}}
            \draw [nodes = {auto, sloped}]
                ({ln(2/3)}, {ln(6)}) node (E) {$\GF(64)$}
                ({ln(1/3)}, {ln(3)}) node (B) {$\GF(8)$}
                ({ln(2)}, {ln(2)}) node (A) {$\GF(4)$}
                (0, 0) node (1) {$\GF(2)$}
                (E) -- node {$2$} (B)
                (E) -- node {$3$} (A)
                (B) -- node {$3$} (1)
                (A) -- node {$2$} (1)
            ;
        \end{tikzpicture}                                    \label{additive}
    \end{equation}
    The only inconvenience is the permutation
    needed to go from \eqref{2^3} to \eqref{3^2}.
    More generally, we have the following result.

    \begin{theorem} [main]                                   \label{thm:main}
        For every prime power $q$ and every degree $n \ge 1$,
        there exists a matrix representation
        $\rho_q^n\colon \GF(q^n) \to \GF(q)^{n\times n}$ that is
        an injective $\GF(q)$-algebra homomorphism, hence a field embedding.
        These maps can be made globally compatible in the sense that,
        for every pair of degrees $m, n \ge 1$, the composition
        \begin{equation}
            \def\.{\cdot}
            \def\r{\rho_q^n}
            \def\M{\bsm{\r&\.\.&\r\\:&::&:\\\r&\.\.&\r}}
            \GF(q^{nm}) \xrightarrow{\rho_{q^n}^m} 
            \GF(q^n)^{m\times m} \xrightarrow{\M}
            (\GF(q)^{n\times n})^{m\times m}
            \cong \GF(q)^{nm\times nm}                          \label{block}
        \end{equation}
        coincides with
        $\rho_q^{nm}\colon \GF(q^{nm}) \to \GF(q)^{nm\times nm}$
        up to row and column permutations.
    \end{theorem}

    The theorem above provides a systematic view of finite fields
    using matrices, allowing self-contained computations
    and a clear visualization of their structure.
    The block structure is particularly useful for understanding subfields.

    A formal proof of Theorem~\ref{thm:main}
    will be given in Section~\ref{sec:main}.
    To demonstrate the strategy, continue with the example above:
    Instead of the power basis $\{1, \epsilon, \dotsc, \epsilon^5\}$
    used by \eqref{companion}, \eqref{2^3} is based on
    \begin{equation}
        \def\e{\epsilon}
        \bma{\e^0 & \e^{21} & \e^9 & \e^{30} & \e^{18} & \e^{39}} =
        \bma{\e^0 & \e^{9} & \e^{18}} \otimes
        \bma{\e^0 & \e^{21}}\in \GF(64)^6                         \label{3*2}
    \end{equation}
    and \eqref{3^2} is based on its permutation
    \begin{equation}
        \def\e{\epsilon}
        \bma{\e^0 & \e^9 & \e^{18} & \e^{21} & \e^{30} & \e^{39}} =
        \bma{\e^0 & \e^{21}} \otimes
        \bma{\e^0 & \e^{9} & \e^{18}} \in \GF(64)^6               \label{2*3}
    \end{equation}
    with $\epsilon$ being a root of $f_2^6 = x^6 + x^4 + x^3 + x + 1$.
    Note that $\beta \coloneqq \epsilon^9$ generates $\GF(8)$ and
    $\alpha \coloneqq \epsilon^{21}$ generates $\GF(4)$, i.e.,
    the multiplicative group side of \eqref{additive} is as follows.
    
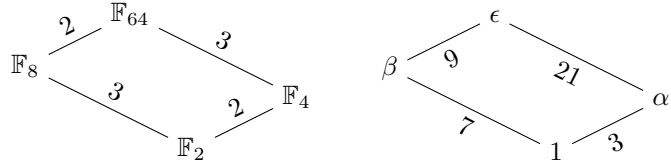
\[
        \begin{tikzpicture} [y=1cm, x=2cm]
            \draw [nodes = {auto, sloped}]
                ({ln(2/3)}, {ln(6)}) node (E) {$\GF(64)$}
                ({ln(1/3)}, {ln(3)}) node (B) {$\GF(8)$}
                ({ln(2)}, {ln(2)}) node (A) {$\GF(4)$}
                (0, 0) node (1) {$\GF(2)$}
                (E) -- node {$2$} (B)
                (E) -- node {$3$} (A)
                (B) -- node {$3$} (1)
                (A) -- node {$2$} (1)
            ;
        \end{tikzpicture}
        \qquad
        \begin{tikzpicture} [y=1cm, x=2cm]
            \draw [nodes = {auto, sloped, '}]
                ({ln(2/3)}, {ln(6)}) node (E) {$\epsilon$}
                ({ln(1/3)}, {ln(3)}) node (B) {$\beta$}
                ({ln(2)}, {ln(2)}) node (A) {$\alpha$}
                ({ln(1)}, {ln(1)}) node (1) {$1$}
                (E) -- node {$9$} (B)
                (E) -- node {$21$} (A)
                (B) -- node {$7$} (1)
                (A) -- node {$3$} (1)
            ;
        \end{tikzpicture}
    \]
    So the bases \eqref{3*2} and \eqref{2*3}
    are just the two ways to Kronecker-product
    $\bma{\alpha^0 & \alpha^1}$ and  $\bma{\beta^0 & \beta^1 & \beta^2}$.
    The row and column permutations needed in the theorem statement
    are just to correct the order in which the Kronecker products are taken.

\section{Preliminaries}

    A \emph{field} is a set with addition and multiplication that have
    additive inverses, multiplicative inverses for nonzero elements,
    associativity for both operators,
    commutativity for both operators,
    and distributivity of multiplication over addition.
    A \emph{finite field} (or a \emph{Galois field})
    is a finite set equipped with field operations.
    The following well-known result classifies all finite fields.

    \begin{fact} [finite field classification]
        There exists a finite field $F$ of size $q$ if and only if
        $q = p^k$ for some prime $p$ and positive exponent $k$.
        Moreover, $F$ is unique up to isomorphism for each such $q$.
        This unique field is usually denoted by
        $\mathbb{F}_q$ or $\mathrm{GF}(q)$.
        See \cite[Theorem~2.5]{LNC09} for a proof.
    \end{fact}

    The theoretical reason that $\GF(p^k)$ is unique is that
    it is the splitting field of the polynomial $x^{p^k} - x$ over $\GF(p)$,
    and the splitting field of a polynomial is unique up to isomorphism.
    However, implementing finite fields in
    a computer algebra system (CAS)\footnote{
    To name a few, see the SageMath documentation
    \url{https://doc.sagemath.org/html/en/reference/finite_rings/sage/rings/finite_rings/finite_field_constructor.html}}%
    \footnote{GAP:
    \url{https://docs.gap-system.org/doc/ref/chap59.html}}%
    \footnote{Macaulay2:
    \url{https://macaulay2.com/doc/Macaulay2/share/doc/Macaulay2/Macaulay2Doc/html/_finite_spfields.html}}%
    \footnote{Magma:
    \url{https://magma.maths.usyd.edu.au/magma/handbook/text/210}}%
    \footnote{Wolfram:
    \url{https://reference.wolfram.com/language/ref/FiniteField.html}}
    involves making choices and breaking the symmetry.
    More precisely, we need to choose polynomials $f_p^k$
    to construct $\GF(p^k)$ as the quotient ring
    $\GF(p)[\kappa] / \langle f_p^k(\kappa) \rangle$.
    Afterwards there are two equally popular options:
    One option is to represent an element of $\GF(p^k)$
    as a polynomial in $\kappa$ of degree less than $k$.
    To do so, we record a $k$-tuple of $\GF(p)$-elements,
    which are essentially $k$ integers in the range $[0, p - 1]$.
    The other option is to represent
    a nonzero element of $\GF(p^k)$ as a power of $\kappa$,
    and so we record an integer in the range $[1, p^k - 1]$,
    and the integer $0$ is reserved for the additive unit of $\GF(p^k)$.
    The latter option turns multiplication into addition of the exponents
    and addition into looking up a table \cite[Exercise~2.8]{LNC09}\footnote{
    Such lookup tables are commonly called Zech logarithm tables.
    See also SageMath's document
    \url{https://doc.sagemath.org/html/en/reference/finite_rings/sage/rings/finite_rings/finite_field_givaro.html}.
    } of size $p^k$.
    When $p^k$ integers fit nicely into a computer's memory,
    this is preferred over multiplying polynomials modulo $f_p^k$.
    The only caveat is that the elements need to be powers of $\kappa$,
    so not every $\kappa$ works.

    \begin{fact} [multiplicative group]
        The \emph{multiplicative group} of a finite field $F$,
        denoted by $F^*$ or $F^\times$, is cyclic.
        An element that generates the whole group
        is said to be \emph{primitive}.
        If one root of an irreducible polynomial is primitive, then all roots
        are, and the polynomial is called a \emph{primitive polynomial}.
        See \cite[Theorem~2.8]{LNC09} for a proof.
    \end{fact}

    The first obstacle we encounter along these implementation
    approaches is when we need to identify subfields in a large finite field.
    When is a finite field a subfield of another finite field?
    How do we derive the embedding map using the two $f$'s?
    To answer these, recall the following results.
    
    \begin{fact} [subfield criteria]
        Fix a prime $p$.
        Let $d$ and $k$ be positive exponents.
        The following are equivalent.
        \begin{itemize}
            \item $d$ divides $k$.
            \item $p^d - 1$ divides $p^k - 1$.
            \item $\GF(p^d)$ is a subfield of $\GF(p^k)$.
            \item $\GF(p^d)^*$ is a subgroup of $\GF(p^k)^*$.
            \item $\GF(p^k)$ is a vector space over $\GF(p^d)$.
        \end{itemize}
        This is a combination of 
        \cite[Lemma~2.1, Theorem~2.6, and Exercise~2.9]{LNC09}.
    \end{fact}

    \begin{corollary} [gcd]
        The gcd of $p^c - 1$ and $p^d - 1$ is $p^{\gcd(c,d)} - 1$.
        The intersection of two subgroups $\GF(p^c)^*$ and $\GF(p^d)^*$
        in a large ambient field is $\GF(p^{\gcd(c,d)})^*$.
        The intersection of two subfields
        $\GF(p^c)$ and $\GF(p^d)$ is $\GF(p^{\gcd(c,d)})$.
    \end{corollary}

    \begin{corollary} [lcm]
        The smallest $p^k - 1$ that is divisible
        by both $p^c - 1$ and $p^d - 1$ is $p^{\lcm(c,d)} - 1$.
        The smallest field-induced group containing both
        $\GF(p^c)^*$ and $\GF(p^d)^*$ is $\GF(p^{\lcm(c,d)})^*$.
        The compositum\footnote{ The compositum of two fields
        is the smallest field that contains both.}
        of $\GF(p^c)$ and $\GF(p^d)$ is $\GF(p^{\lcm(c,d)})$.
    \end{corollary}

    Suppose that $\kappa$ is a generator of $\GF(p^k)^*$.
    We see, from the given fact, that $\kappa^{(p^k-1)/(p^d-1)}$
    generates the subgroup of $\GF(p^k)^*$ of size $p^d - 1$,
    and so it generates the copy of the subfield $\GF(p^d)$ in $\GF(p^k)$.
    That is one mathematically correct way to implement $\GF(p^d)$
    as an individual field, but not a reasonable one.
    The constructions in CASs go in the other way around:
    We first make $\delta$ a generator of $\GF(p^d)^*$
    by choosing a suitable $f_p^d$, and when we construct $\GF(p^k)$,
    we choose $f_p^k$ so that $\kappa$ generates $\GF(p^k)^*$
    and $\kappa^{(p^k-1)/(p^d-1)} = \delta$.
    This is called the \emph{norm-compatibility} condition in CAS literature
    because an equivalent way to state it is that $\norm(\kappa) = \delta$,
    where the field norm maps from $\GF(p^k)$ to $\GF(p^d)$.

    More generally, if $k$ has another divisor, say $c$,
    then $f_p^k$ should be chosen such that $\kappa^{(p^k-1)/(p^c-1)}$
    also coincides with $\epsilon$, the generator of $\GF(p^c)^*$.
    This motivates the definition of the Conway polynomials.

    \begin{theorem} [Conway polynomials]                   \label{thm:conway}
        Fix a prime $p$.
        Then there exists a family of
        irreducible polynomials $f_p^k \in \GF(p)[x]$ such that
        any root of $f_p^k$ generates $\GF(p^k)^*$ and
        $f_p^k(x)$ divides $f_p^d(x^{(p^k - 1) / (p^d - 1)})$
        whenever $d$ divides $k$.
        Note that the choice is not unique, and \emph{Conway polynomials}
        refer to the ones that are lexicographically minimal.
        See \cite{Nic88} for a proof.
        See \cite{Lub23} for a modern (2023) alternative.
    \end{theorem}


    The existing proof of Theorem~\ref{thm:conway} shares
    a similar spirit as our proof of Theorem~\ref{thm:main},
    so we illustrate by example the proof of
    Theorem~\ref{thm:conway} in Appendix~\ref{sec:conway}.
    The remainder of this note is organized as follows:
    In Section~\ref{sec:12}, we give a degree-$12$ example to demonstrate
    that our method is not limited to square-free degrees.
    In Section~\ref{sec:main}, we prove Theorem~\ref{thm:main}.
    Then in Section~\ref{sec:frobenius}, we discuss how to use
    a similar mechanism to visualize the Frobenius automorphism.

\section{One More Example With Degree Twelve}                  \label{sec:12}

    Before we prove Theorem~\ref{thm:main}, let us look at
    the extension of $\GF(2)$ of degree $12 = 2 \cdot 2 \cdot 3$.
    We take $(a, b, c, e, k) = (2, 3, 4, 6, 12)$, and let
    $\alpha$, $\beta$, $\gamma$, $\epsilon$, and $\kappa$
    be the generators of (the multiplicative groups of)
    $\GF(4)$, $\GF(8)$, $\GF(16)$, $\GF(64)$, and $\GF(4096)$, respectively.
    Their subfield relations and
    norm-compatibility conditions are as below.
    \[
        \begin{tikzpicture} [y=1cm, x=2cm]
            \draw [nodes = {auto, sloped}]
                ({ln(4/3)}, {ln(12)}) node (F12) {$\GF(4096)$}
                ({ln(2/3)}, {ln(6)}) node (F6) {$\GF(64)$}
                ({ln(4)}, {ln(4)}) node (F4) {$\GF(16)$}
                ({ln(1/3)}, {ln(3)}) node (F3) {$\GF(8)$}
                ({ln(2)}, {ln(2)}) node (F2) {$\GF(4)$}
                ({ln(1)}, {ln(1)}) node (F1) {$\GF(2)$}
                (F12) -- node {$2$} (F6)
                (F12) -- node {$3$} (F4)
                (F6) -- node {$2$} (F3)
                (F6) -- node {$3$} (F2)
                (F4) -- node {$2$} (F2)
                (F3) -- node {$3$} (F1)
                (F2) -- node {$2$} (F1)
            ;
        \end{tikzpicture}
        \qquad
        \begin{tikzpicture} [y=1cm, x=2cm]
            \draw [nodes = {auto, sloped, '}]
                ({ln(4/3)}, {ln(12)}) node (F12) {$\kappa$}
                ({ln(2/3)}, {ln(6)}) node (F6) {$\epsilon$}
                ({ln(4)}, {ln(4)}) node (F4) {$\gamma$}
                ({ln(1/3)}, {ln(3)}) node (F3) {$\beta$}
                ({ln(2)}, {ln(2)}) node (F2) {$\alpha$}
                ({ln(1)}, {ln(1)}) node (F1) {$1$}
                (F12) -- node {$65$} (F6)
                (F12) -- node {$273$} (F4)
                (F6) -- node {$9$} (F3)
                (F6) -- node {$21$} (F2)
                (F4) -- node {$5$} (F2)
                (F3) -- node {$7$} (F1)
                (F2) -- node {$3$} (F1)
            ;
        \end{tikzpicture}
    \]
    While \eqref{2*3} and \eqref{3*2} are
    two subfield chains from $\GF(64)$ to $\GF(2)$,
    there are three subfield chains from $\GF(4096)$ to $\GF(2)$.
    Hence, we consider three different bases.
    
    The first basis of $\GF(4096) / \GF(2)$ is
    \begin{equation}
        \bma{1 & \beta & \beta^2} \otimes
        \bma{1 & \gamma} \otimes
        \bma{1 & \alpha}
        \in \GF(4096)^{12}.                                     \label{3*2*2}
    \end{equation}
    Because $\alpha$, $\beta$, and $\gamma$ are
    $\kappa^{(2^{12}-1)/(2^2-1)} = \kappa^{1365}$,
    $\kappa^{(2^{12}-1)/(2^3-1)} = \kappa^{585}$, and
    $\kappa^{(2^{12}-1)/(2^4-1)} = \kappa^{273}$, respectively,
    \eqref{3*2*2} is the same as
    \[
        \arraycolsep4pt
        \bma{
            \kappa^0      & \kappa^{1365} & \kappa^{273}  &
            \kappa^{1638} & \kappa^{585}  & \kappa^{1950} &
            \kappa^{858}  & \kappa^{2223} & \kappa^{1170} &
            \kappa^{2535} & \kappa^{1443} & \kappa^{2808}
        }.
    \]
    From this and
    $\kappa^{12} + \kappa^{11} + \kappa^{10} + \kappa^4 + 1 = 0$
    we can compute
    \begin{equation*}
        \rho_2^{12}(\kappa) =
        \vscale{\vscale{\left[
            \begin{array} {cc|cc||cc|cc||cc|cc}
                1 & 0 & 1 & 1 & 1 & 0 & 1 & 1 & 1 & 1 & 0 & 0 \\
                0 & 1 & 1 & 0 & 0 & 1 & 1 & 0 & 1 & 0 & 0 & 0 \\\hline
                0 & 1 & 0 & 1 & 0 & 1 & 0 & 1 & 0 & 0 & 1 & 1 \\
                1 & 1 & 1 & 1 & 1 & 1 & 1 & 1 & 0 & 0 & 1 & 0 \\\hline\hline
                1 & 1 & 0 & 0 & 0 & 0 & 0 & 0 & 0 & 1 & 1 & 1 \\
                1 & 0 & 0 & 0 & 0 & 0 & 0 & 0 & 1 & 1 & 1 & 0 \\\hline
                0 & 0 & 1 & 1 & 0 & 0 & 0 & 0 & 0 & 1 & 1 & 0 \\
                0 & 0 & 1 & 0 & 0 & 0 & 0 & 0 & 1 & 1 & 0 & 1 \\\hline\hline
                1 & 0 & 1 & 1 & 1 & 1 & 0 & 0 & 0 & 0 & 0 & 0 \\
                0 & 1 & 1 & 0 & 1 & 0 & 0 & 0 & 0 & 0 & 0 & 0 \\\hline
                0 & 1 & 0 & 1 & 0 & 0 & 1 & 1 & 0 & 0 & 0 & 0 \\
                1 & 1 & 1 & 1 & 0 & 0 & 1 & 0 & 0 & 0 & 0 & 0
            \end{array}
        \right]}}.
    \end{equation*}
    We can also use the first two terms of
    \eqref{3*2*2}---$\bma{1 & \beta & \beta^2} \otimes \bma{1 & \gamma}
    $---and the first term of
    \eqref{3*2*2}---$\bma{1 & \beta & \beta^2}$---to
    construct $\rho_4^6$ and $\rho_{16}^3$:
    \begin{equation*}
        \rho_4^6(\kappa) =
        \vscale{\left[
            \def\0{       0}
            \def\1{\alpha^1}
            \def\2{\alpha^2}
            \def\3{\alpha^3}
            \begin{array} {cc|cc|cc}
                \3 & \2 & \3 & \2 & \2 & \0 \\
                \1 & \1 & \1 & \1 & \0 & \2 \\\hline
                \2 & \0 & \0 & \0 & \1 & \2 \\
                \0 & \2 & \0 & \0 & \1 & \3 \\\hline
                \3 & \2 & \2 & \0 & \0 & \0 \\
                \1 & \1 & \0 & \2 & \0 & \0
            \end{array}
        \right]}
        \qquad
        \rho_{16}^3(\kappa) =
        \bma{
            \gamma^8    & \gamma^8    & \gamma^{10} \\
            \gamma^{10} &        0    & \gamma^2    \\
            \gamma^8    & \gamma^{10} &        0
        }
    \end{equation*}
    Our structure theorem says that these three matrices are actually one.
    For instance, the lower-left block of $\rho_2^{12}(\kappa)$ is
    $\lss{0 & 1 \\ 1 & 1}$, which is $\rho_2^2(\alpha^1)$,
    where $\alpha^1$ is the lower-left entry of $\rho_4^6(\kappa)$.
    Also the lower-left block of $\rho_4^6(\kappa)$ is
    $\lsm{\alpha^3 & \alpha^2 \\ \alpha^1 & \alpha^1}$,
    which is $\rho_4^2(\gamma^8)$,
    where $\gamma^8$ is the lower-left entry of $\rho_{16}^3(\kappa)$.
    In other words, \eqref{3*2*2} helps visualize
    the subfield chain $\GF(4096) / \GF(16) / \GF(4) / \GF(2)$. 

    The second basis of $\GF(4096) / \GF(2)$ we consider is
    \begin{equation}
        \bma{1 & \gamma} \otimes
        \bma{1 & \beta & \beta^2} \otimes
        \bma{1 & \alpha}
        \in \GF(4096)^{12}.                                     \label{2*3*2}
    \end{equation}
    Under this basis, the matrix representation of $\kappa$ is

    \begin{equation*}
        \rho_2^{12}(\kappa) =
        \vscale{\vscale{\left[
            \begin{array} {cc|cc|cc||cc|cc|cc}
                1 & 0 & 1 & 0 & 1 & 1 & 1 & 1 & 1 & 1 & 0 & 0 \\
                0 & 1 & 0 & 1 & 1 & 0 & 1 & 0 & 1 & 0 & 0 & 0 \\\hline
                1 & 1 & 0 & 0 & 0 & 1 & 0 & 0 & 0 & 0 & 1 & 1 \\
                1 & 0 & 0 & 0 & 1 & 1 & 0 & 0 & 0 & 0 & 1 & 0 \\\hline
                1 & 0 & 1 & 1 & 0 & 0 & 1 & 1 & 0 & 0 & 0 & 0 \\
                0 & 1 & 1 & 0 & 0 & 0 & 1 & 0 & 0 & 0 & 0 & 0 \\\hline\hline
                0 & 1 & 0 & 1 & 0 & 0 & 0 & 1 & 0 & 1 & 1 & 1 \\
                1 & 1 & 1 & 1 & 0 & 0 & 1 & 1 & 1 & 1 & 1 & 0 \\\hline
                0 & 0 & 0 & 0 & 0 & 1 & 1 & 1 & 0 & 0 & 1 & 0 \\
                0 & 0 & 0 & 0 & 1 & 1 & 1 & 0 & 0 & 0 & 0 & 1 \\\hline
                0 & 1 & 0 & 0 & 0 & 0 & 0 & 1 & 1 & 1 & 0 & 0 \\
                1 & 1 & 0 & 0 & 0 & 0 & 1 & 1 & 1 & 0 & 0 & 0
            \end{array}
        \right]}}.
    \end{equation*}
    Now using the prefixes of
    \eqref{2*3*2}---$\bma{1 & \gamma} \otimes \bma{1 & \beta & \beta^2}$
    and $\bma{1 & \gamma}$---we can construct
    \begin{equation*}
        \rho_4^6(\kappa) =
        \vscale{\left[
            \def\0{       0}
            \def\1{\alpha^1}
            \def\2{\alpha^2}
            \def\3{\alpha^3}
            \begin{array} {ccc|ccc}
                \3 & \3 & \2 & \2 & \2 & \0 \\
                \2 & \0 & \1 & \0 & \0 & \2 \\
                \3 & \2 & \0 & \2 & \0 & \0 \\\hline
                \1 & \1 & \0 & \1 & \1 & \2 \\
                \0 & \0 & \1 & \2 & \0 & \3 \\
                \1 & \0 & \0 & \1 & \2 & \0 \\
            \end{array}
        \right]}
        \qquad
        \rho_{64}^2(\kappa) =
        \bma{
            \epsilon^1    & \epsilon^{33} \\
            \epsilon^{12} & \epsilon^{29}
        }.
    \end{equation*}
    In particular,
    $
    \def\0{       0}
    \def\1{\alpha^1}
    \def\2{\alpha^2}
    \def\3{\alpha^3}
    \lsm{
        \1 & \1 & \0 \\
        \0 & \0 & \1 \\
        \1 & \0 & \0
    }$
    is the lower-left corner of $\rho_4^6(\kappa)$, which is
    $\rho_4^2(\epsilon^{12})$, where $\epsilon^{12}$
    is the lower-left corner of $\rho_{64}^2(\kappa)$.
    In other words, \eqref{2*3*2} helps visualize
    the subfield chain $\GF(4096) / \GF(64) / \GF(4) / \GF(2)$.

    The third basis of $\GF(4096) / \GF(2)$ we consider is
    \begin{equation}
        \bma{1 & \gamma} \otimes
        \bma{1 & \alpha} \otimes
        \bma{1 & \beta & \beta^2}
        \in \GF(4096)^{12}.                                     \label{2*2*3}
    \end{equation}
    Under this basis, the matrix representation of $\kappa$ is
    \begin{equation*}
        \rho_2^{12}(\kappa) =
        \vscale{\vscale{\left[
            \begin{array} {ccc|ccc||ccc|ccc}
                1 & 1 & 1 & 0 & 0 & 1 & 1 & 1 & 0 & 1 & 1 & 0 \\
                1 & 0 & 0 & 1 & 0 & 1 & 0 & 0 & 1 & 0 & 0 & 1 \\
                1 & 1 & 0 & 0 & 1 & 0 & 1 & 0 & 0 & 1 & 0 & 0 \\\hline
                0 & 0 & 1 & 1 & 1 & 0 & 1 & 1 & 0 & 0 & 0 & 0 \\
                1 & 0 & 1 & 0 & 0 & 1 & 0 & 0 & 1 & 0 & 0 & 0 \\
                0 & 1 & 0 & 1 & 0 & 0 & 1 & 0 & 0 & 0 & 0 & 0 \\\hline\hline
                0 & 0 & 0 & 1 & 1 & 0 & 0 & 0 & 1 & 1 & 1 & 1 \\
                0 & 0 & 0 & 0 & 0 & 1 & 1 & 0 & 1 & 1 & 0 & 0 \\
                0 & 0 & 0 & 1 & 0 & 0 & 0 & 1 & 0 & 1 & 1 & 0 \\\hline
                1 & 1 & 0 & 1 & 1 & 0 & 1 & 1 & 1 & 1 & 1 & 0 \\
                0 & 0 & 1 & 0 & 0 & 1 & 1 & 0 & 0 & 0 & 0 & 1 \\
                1 & 0 & 0 & 1 & 0 & 0 & 1 & 1 & 0 & 1 & 0 & 0
            \end{array}
        \right]}}.
    \end{equation*}
    Now using the prefixes of
    \eqref{2*2*3}---$\bma{1 & \gamma} \otimes \bma{1 & \alpha}$ and
    $\bma{1 & \gamma}$---we obtain
    \begin{equation*}
        \rho_8^4(\kappa) =
        \vscale{\left[
            \def\0{       0}
            \def\1{\beta^1}
            \def\5{\beta^5}
            \def\6{\beta^6}
            \begin{array} {cc|cc}
                \5 & \1 & \6 & \6 \\
                \1 & \6 & \6 & \0 \\\hline
                \0 & \6 & \1 & \5 \\
                \6 & \6 & \5 & \6
            \end{array}
        \right]}
        \qquad
        \rho_{64}^2(\kappa) =
        \bma{
            \epsilon^1    & \epsilon^{33} \\
            \epsilon^{12} & \epsilon^{29}
        }.
    \end{equation*}
    In particular,
    $
    \def\0{       0}
    \def\1{\beta^1}
    \def\6{\beta^6}
    \bsm{
        \0 & \6 \\
        \6 & \6
    }$
    is the lower-left corner of $\rho_8^4(\kappa)$, which is
    $\rho_8^2(\epsilon^{12})$, where $\epsilon^{12}$
    is the lower-left corner of $\rho_{64}^2(\kappa)$.
    In other words, \eqref{2*2*3} helps visualize
    the subfield chain $\GF(4096) / \GF(64) / \GF(8) / \GF(2)$.

    Summary of strategy:
    As there are three ways to factorize,
    $12 = 2 \cdot 2 \cdot 3 = 2 \cdot 3 \cdot 2 = 3 \cdot 2 \cdot 2$,
    there are three different subfield chains from $\GF(4096)$ to $\GF(2)$.
    Each chain corresponds to a different order
    of Kronecker products of the bases
    $\bma{1 & \alpha}$, $\bma{1 & \beta & \beta^2}$, and $\bma{1 & \gamma}$.
    Note that, in all three of \eqref{3*2*2}, \eqref{2*3*2}, and
    \eqref{2*2*3}, $\bma{1 & \gamma}$ always appears to the left of
    $\bma{1 & \alpha}$ because, from top to bottom,
    $\GF(16)$ always appears before $\GF(4)$.

\subsection{One more example on degree thirty}

    Before we prove Theorem~\ref{thm:main}, let us briefly go over
    degree $30$, a product of three distinct primes $2$, $3$, and $5$.
    Cf.\ \cite[Example~2.7]{LNC09}.
    \[
        \begin{tikzpicture} [y=1cm, x=2cm]
            \draw [nodes = {auto, sloped}]
                ({ln(2/3)}, {ln(30)}) node (F30) {$\GF(2^{30})$}
                ({ln(1/3)}, {ln(15)}) node (F15) {$\GF(32768)$}
                ({ln(2)}, {ln(10)}) node (F10) {$\GF(1024)$}
                ({ln(2/3)}, {ln(6)}) node (F6) {$\GF(64)$}
                ({ln(1)}, {ln(5)}) node (F5) {$\GF(32)$}
                ({ln(1/3)}, {ln(3)}) node (F3) {$\GF(8)$}
                ({ln(2)}, {ln(2)}) node (F2) {$\GF(4)$}
                ({ln(1)}, {ln(1)}) node (F1) {$\GF(2)$}
                (F30) -- node {$2$} (F15)
                (F30) -- node {$3$} (F10)
                (F30) -- node {$5$} (F6)
                (F15) -- node {$3$} (F5)
                (F15) -- node {$5$} (F3)
                (F10) -- node {$2$} (F5)
                (F10) -- node {$5$} (F2)
                (F6) -- node {$3$} (F2)
                (F6) -- node {$2$} (F3)
                (F5) -- node {$5$} (F1)
                (F3) -- node {$3$} (F1)
                (F2) -- node {$2$} (F1)
            ;
        \end{tikzpicture}
        \qquad
        \begin{tikzpicture} [y=1cm, x=2cm]
            \draw [nodes = {auto, sloped, '}]
                ({ln(2/3)}, {ln(30)}) node (F30) {$\bullet$}
                ({ln(1/3)}, {ln(15)}) node (F15) {$\bullet$}
                ({ln(2)}, {ln(10)}) node (F10) {$\bullet$}
                ({ln(2/3)}, {ln(6)}) node (F6) {$\epsilon$}
                ({ln(1)}, {ln(5)}) node (F5) {$\delta$}
                ({ln(1/3)}, {ln(3)}) node (F3) {$\beta$}
                ({ln(2)}, {ln(2)}) node (F2) {$\alpha$}
                ({ln(1)}, {ln(1)}) node (F1) {$1$}
                (F30) -- node {$32769$} (F15)
                (F30) -- node {$1049601$} (F10)
                (F30) -- node {$ $} (F6)
                (F15) -- node {$1057$} (F5)
                (F15) -- node {$4681$} (F3)
                (F10) -- node {$33$} (F5)
                (F10) -- node {$341$} (F2)
                (F6) -- node {$21$} (F2)
                (F6) -- node {$9$} (F3)
                (F5) -- node {$31$} (F1)
                (F3) -- node {$7$} (F1)
                (F2) -- node {$3$} (F1)
            ;
        \end{tikzpicture}
    \]
    Here $\delta$ generates $\GF(32)^*$.
    For this case, we use these bases
    \[
        A \coloneqq \bma{1 & \alpha}, \qquad
        B \coloneqq \bma{1 & \beta & \beta^2}, \qquad
        D \coloneqq \bma{1 & \delta & \delta^2 & \delta^3 & \delta^4}
    \]
    as building blocks.
    There are six ways to arrange $A$, $B$, and $D$:
    \begin{itemize}
        \item $A \otimes B \otimes D$ works for 
            $\GF(2^{30}) / \GF(32768) / \GF(32) / \GF(2)$.
        \item $A \otimes D \otimes B$ works for
            $\GF(2^{30}) / \GF(32768) / \GF(8) / \GF(2)$.
        \item $B \otimes A \otimes D$ works for
            $\GF(2^{30}) / \GF(1024) / \GF(32) / \GF(2)$.
        \item $B \otimes D \otimes A$ works for
            $\GF(2^{30}) / \GF(1024) / \GF(4) / \GF(2)$.
        \item $D \otimes A \otimes B$ works for
            $\GF(2^{30}) / \GF(64) / \GF(8) / \GF(2)$.
        \item $D \otimes B \otimes A$ works for
            $\GF(2^{30}) / \GF(64) / \GF(4) / \GF(2)$.
    \end{itemize}

\section{Proof of Theorem~\ref{thm:main}}                    \label{sec:main}

    The proof consists of several steps.
    Each step corresponds to a subsection below.

\paragraph{Step 1.}
    We show that a basis $N \in \GF(q^n)^n$ of $\GF(q^n) / \GF(q)$
    induces a matrix representation $\GF(q^n) \to \GF(q)^{n\times n}$
    that is an injective $\GF(q)$-algebra homomorphism and a field embedding.

\paragraph{Step 2.}
    We show that if $M \in \GF(q^{nm})^m$ is a basis chosen for
    $\GF(q^{nm}) / \GF(q^n)$, then $M \otimes N \in \GF(q^{nm})^{nm}$
    is a basis of $\GF(q^{nm}) / \GF(q)$ and induces the block
    structure of the maps we demonstrated in Section~\ref{sec:12}.

\paragraph{Step 3.}
    We describe how to construct the next basis when
    the incremental degree $m$ is a prime.

\paragraph{Step 4.}
    We show that the resulting basis of $\GF(q^{nm}) / \GF(q)$,
    up to permutation, does not depend on factorization.

\subsection{Matrix representation from a basis}

    Let $\NN \coloneqq \GF(q^n)$.
    For every $\xi \in \NN$,
    multiplication by $\xi$ defines a $\GF(q)$-linear map
    $\mu_\xi\colon \NN \to \NN$ by $\mu_\xi(\eta) = \xi\eta$.
    Let $N \coloneqq \bma{\nu_0 & \nu_1 & \cdots & \nu_{n-1}} \in \NN^n$
    be a basis of $\NN / \GF(q)$.
    Every linear map has a matrix representation once a basis is chosen:
    Let $\rho_N(\xi)$ be the matrix form of $\mu_\xi$ with respect to $N$.
    That is, if
    $\xi \nu_j = x_{0j} \nu_0 + \dotsc + x_{n-1,j} \nu_{n-1}$
    is how $\xi \nu_j$ is expressed in the basis $N$, then $x_{ij}$,
    for $0 \le i, j < n$, is the $(i, j)$-entry of $\rho_N(\xi)$.

    Multiplication in $\NN$ is
    distributive---$(\xi + \eta)\zeta = \xi\zeta + \eta\zeta$---so
    $\rho_N(\xi + \eta) = \rho_N(\xi) + \rho_N(\eta)$.
    Multiplication is also
    associative---$(\xi\eta)\zeta = \xi(\eta\zeta)$---so
    $\rho_N(\xi\eta) = \rho_N(\xi) \rho_N(\eta)$.
    We also have $\rho_N(1) = I_n$ because $1\nu_i = \nu_i$
    for every $0 \le i < n$.
    These three properties imply that $\rho_N$ is a
    ring homomorphism from $\NN$ to $\GF(q)^{n\times n}$.
    This homomorphism is nonzero,
    and a nonzero ring homomorphism from a field is injective.
    This ensures that $\rho_N$ is at least a field embedding.

    For a scalar $y \in \GF(q)$, multiplication by $y$
    sends every basis vector $\nu_i$ to $y \nu_i$
    with no cross components, and hence $\rho_N(y) = y \cdot I_n$.
    Therefore, for every $\xi \in \NN$, we see that $\rho_N(y \xi) =
    \rho_N(y) \rho_N(\xi) = (y \cdot I_n) \rho_N(\xi) = y \cdot \rho_N(\xi)$.
    Hence $\rho_N$ is an $\GF(q)$-algebra homomorphism.

\subsection{Representation from Kronecker product of bases}

    Let $\NN \coloneqq \GF(q^n)$ and $\MM \coloneqq \GF(q^{nm})$.
    Let $N \coloneqq \bma{\nu_0 & \cdots & \nu_{n-1}} \in \NN^n$
    be a basis of $\NN / \GF(q)$, and let
    $M \coloneqq \bma{\mu_0 & \cdots & \mu_{m-1}} \in \MM^m$
    be a basis of $\MM / \NN$.
    Observe that the Kronecker product
    \[
        M \otimes N =
        \bma{
            \mu_0 \nu_0 & \cdots & \mu_0 \nu_{n-1}
            & \bullet & \bullet & \bullet &
            \mu_{m-1} \nu_0 & \cdots & \mu_{m-1} \nu_{n-1}
        }
    \]
    forms a basis of $\MM / \GF(q)$.
    This is because every element of $\MM$ can be written
    as a linear combination of the $\mu_j$ with coefficients in $\NN$,
    and every coefficient in $\NN$ can be written as a
    linear combination of $\nu_i$ with coefficients in $\GF(q)$.

    To see the block structure explicitly, we use
    $0 \le \ii, \jj < m$ for the block indices and
    $0 \le i, j < n$ for the indices inside each block.
    Let $\rho_N$ be the matrix representation of $\NN / \GF(q)$ induced by
    $N$, and let $\rho_M$ be that of $\MM / \NN$ induced by $M$.
    For a fixed $\Xi \in \MM$, write
    $\rho_M(\Xi) = [\xi_{\ii\jj}]_{\ii\jj} \in \NN^{m\times m}$.
    By definition,
    \[ \Xi \mu_\jj = \sum_\ii \xi_{\ii\jj} \mu_\ii .  \]
    Also, for each coefficient $\xi_{\ii\jj}\in\NN$,
    the definition of $\rho_N$ says that
    \[ \xi_{\ii\jj}\nu_j = \sum_i \rho_N(\xi_{\ii\jj})_{ij}\nu_i. \]
    Chaining these two leads to 
    \[
        (\Xi\mu_\jj) \nu_j
        = \sum_\ii (\xi_{\ii\jj} \mu_\ii) \nu_j
        = \sum_\ii \mu_\ii (\xi_{\ii\jj} \nu_j)
        = \sum_\ii \mu_\ii \sum_i \rho_N(\xi_{\ii\jj})_{ij}  \nu_i .
    \]
    Now write
    $\rho_{M\otimes N}(\Xi) = [x_{\ii n+i,\jj n+j}]_{\ii n+i,\jj n+j}
    \in \GF(q)^{nm\times nm}$, i.e.,
    \[
        \Xi (\mu_\jj \nu_j)
        = \sum_{\ii n+i} x_{\ii n+i,\jj n + j} \mu_\ii \nu_i.
    \]
    Comparing how $\Xi$ acts on $\mu_\jj \nu_j$ in the two ways above,
    we see that
    $x_{\ii n+i,\jj n+j} = \bigl(\rho_N(\xi_{\ii\jj})\bigr)_{ij}$.
    This implies that $\rho_{M\otimes N}$ possesses
    the block structure \eqref{block} required in Theorem~\ref{thm:main},
    provided that the bases involved possess the Kronecker-product structure.

\subsection{The bases for prime degree extensions}

    It remains to explain which bases we choose.
    Fix a prime $p$.
    Fix, once and for all, compatible generators for the fields $\GF(p^k)$:
    For every $k$, let $\omega_k$ generate $\GF(p^k)^*$, and require
    that $\omega_k^{(p^k-1)/(p^d-1)} = \omega_d$ whenever $d \mid k$.
    This is the compatibility supplied by Conway polynomials, or by any
    compatible primitive system of defining polynomials.

    Suppose we have already constructed a basis of
    $\GF(p^d) / \GF(p)$, and we want to extend from
    $\GF(p^d)$ to $\GF(p^k)$, where $r \coloneqq k/d$ is prime.
    Write $d = r^s t$ with $\gcd(r, t) = 1$,
    i.e., $d$ already contains $s$ copies of $r$ and $k$ contains one more.
    Let $\sigma \coloneqq \omega_{r^{s+1}} \in \GF(p^{r^{s+1}})$.
    Then we have the following diamond.
    
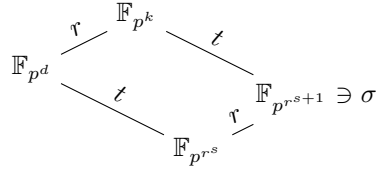
\begin{equation*}
        \begin{tikzpicture} [y=1cm, x=2cm]
            \tikzset{baseline={([yshift=-.5ex]current bounding box.center)}}
            \draw [nodes = {auto, sloped}]
                ({ln(2/3)}, {ln(6)}) node (E) {$\GF(p^k)$}
                ({ln(1/3)}, {ln(3)}) node (B) {$\GF(p^d)$}
                ({ln(2)}, {ln(2)}) node (A)
                {\kern1em $\GF(p^{r^{s+1}}) \ni \sigma$}
                (0, 0) node (1) {$\GF(p^{r^s})$}
                (E) -- node {$r$} (B)
                (E) -- node {$t$} (A)
                (B) -- node {$t$} (1)
                (A) -- node {$r$} (1)
            ;
        \end{tikzpicture}
    \end{equation*}
    The intersection of the two middle fields is the bottom field,
    and the compositum of the two middle fields is the top field.
    In particular,
    \[
        \GF(p^d) \cap \GF(p^{r^{s+1}}) = \GF(p^{\gcd=r^s}), \qquad
        \GF(p^d)[\sigma] = \GF(p^{\lcm=k}).
    \]
    Hence $\sigma$ has degree $r$ over $\GF(p^d)$, and
    $\bma{1 & \sigma & \cdots & \sigma^{r-1}}$
    is a basis of $\GF(p^k) / \GF(p^d)$.
    
    This prime-degree step has appeared multiple times before.
    For instance, $\bma{1 & \beta & \beta^2}$
    is used for extensions of degree $3$ in \eqref{3*2},
    \eqref{2*3}, \eqref{3*2*2}, \eqref{2*3*2}, and \eqref{2*2*3}.
    $\bma{1 & \alpha}$ is used for extensions of degree $2$
    when it is the lowest extension of degree $2$ in the chain;
    when it is not, $\bma{1 & \gamma}$ is used instead.

\subsection{The independence of the factorization}

    In the previous subsection we declared that each extension of prime
    degree uses a basis of the form $\bma{1 & \sigma & \cdots & \sigma^{r-1}}$.
    Two subsections ago we also clarified that the Kronecker product
    gives a basis for the compositum of two extensions,
    and the block structure follows.
    It remains to explain why the Kronecker product of the bases
    does not depend on the order of the prime factors.

    The fundamental reason is that changing the order of a Kronecker product
    only permutes the entries of the resulting basis vector,
    so the only thing that matters is the \emph{multiset} of bases
    whose Kronecker product we take, not the order in which we take it.
    Now, the first time a prime $r$ appears in the chain, it contributes
    $\bma{1 & \omega_r & \cdots & \omega_r^{r-1}}$.
    If $r^2$ divides $k$, then the second occurrence of $r$ contributes
    $\bma{1 & \omega_{r^2} & \cdots & \omega_{r^2}^{r-1}}$.
    If even $r^3$ divides $k$, then the third occurrence contributes
    $\bma{1 & \omega_{r^3} & \cdots & \omega_{r^3}^{r-1}}$,
    and so on.
    This confirms that the bases are algorithmically determined
    by the prime powers $r^a$ dividing $k$,
    rather than by the order in which the prime factors are adjoined.
    This finishes the proof of Theorem~\ref{thm:main}.

\subsection{Consequences of the main theorem}

    Because the maps $\rho_q^n$ are, by construction, matrix representations
    of the $\GF(q)$-linear transformations, the field trace and field norm
    are simply the matrix trace and matrix determinant, respectively.

    \begin{corollary} [trace and norm]
        The $\rho_q^n$ described in Theorem~\ref{thm:main} satisfy
        \[
            \tr(\xi) = \xi + \xi^q + \dotsb + \xi^{q^{n-1}}
            = \tr(\rho_q^n(\xi)) \in \GF(q),
        \]
        where the left-hand side is the field trace from $\GF(q^n)$ and
        the right-hand side is the matrix trace from $\GF(q)^{n\times n}$.
        Similarly,
        \[
            \norm(\xi) = \xi \cdot \xi^q \dotsm \xi^{q^{n-1}}
            = \det(\rho_q^n(\xi)) \in \GF(q),
        \]
        where the left-hand side is the field norm from $\GF(q^n)$ and the
        right-hand side is the matrix determinant from $\GF(q)^{n\times n}$.
    \end{corollary}

    Trace and norm/determinant are coefficients
    of the characteristic polynomials, so it is not
    a surprise that the previous corollary generalizes.

    \begin{corollary} [minimal and characteristic polynomial]
        The $\rho_q^n$ described in Theorem~\ref{thm:main} satisfy
        \[ \minpoly(\xi) = \minpoly(\rho_q^n(\xi)) \in \GF(q)[x], \]
        where the left-hand side is the minimal polynomial
        for field extensions and the right-hand side
        is the minimal polynomial for matrices.
        Similarly,
        \[ \charpoly(\xi) = \charpoly(\rho_q^n(\xi)) \in \GF(q)[x], \]
        where the left-hand side is the characteristic polynomial
        for field extensions and the right-hand side
        is the characteristic polynomial for matrices.
    \end{corollary}

    \begin{proof}
        Since $\rho_q^n$ is an injective $\GF(q)$-algebra homomorphism,
        every polynomial $f(x) \in \GF(q)[x]$ commutes with it:
        $\rho_q^n(f(\xi)) = f(\rho_q^n(\xi))$.
        Thus $f(\xi) = 0$ if and only if $f(\rho_q^n(\xi)) = 0 \cdot I_n$,
        which proves the statement for minimal polynomials.
        The characteristic polynomial of $\xi$ over $\GF(q)$ is the
        characteristic polynomial of the $\GF(q)$-linear map
        $\mu_\xi\colon \eta \mapsto \xi\eta$.
        Since $\rho_q^n(\xi)$ is the matrix form of $\mu_\xi$,
        the characteristic polynomials also coincide.
    \end{proof}

    Interesting things happen when $\xi$ is in the base field $\GF(q)$.
    If $\xi \in \GF(q)$, the minimal polynomial of $\xi$
    has degree one: $\minpoly(\xi) = x - \xi$.
    This implies that the matrix $\rho_q^n(\xi)$ must also satisfy
    $x - \xi = 0$, which leads to $\rho_q^n(\xi) - \xi \cdot I_n = 0$,
    where $I_n$ is the $n \times n$ identity matrix.
    Now combine this fact with the block structure:
    For $\alpha \in \GF(4)$, we have
    \[
        \rho_4^3(\alpha) = \bma{\alpha \\& \alpha \\&& \alpha}
        \in \GF(4)^{3\times3}.
    \]
    We then apply $\rho_2^2$ to the resulting matrix to get
    \[
        \rho_2^6(\alpha) = \rho_2^2(\rho_4^3(\alpha)) =
        \bma{\rho_2^2(\alpha) \\& \rho_2^2(\alpha) \\&& \rho_2^2(\alpha)}
        \in \GF(2)^{6\times6}.
    \]
    In other words, diagonal matrices correspond to base-field elements,
    while block-diagonal matrices correspond to subfield elements.
    The size of the blocks reveals
    the smallest subfield containing the element.
    This is exactly what \eqref{I_3} and \eqref{I_2} want to demonstrate.

    \begin{corollary} [block diagonal]
        If $\xi$ is in $\GF(q^n)$ but treated as an element of $\GF(q^{nm})$,
        then, up to permutations,
        \[
            \def\ddot{\raisebox{0.6ex}{$.$}.}
            \rho_q^{nm}(\xi) =
            \bma{\rho_q^n(\xi) \\& \ddot \\&& \rho_q^n(\xi)}
            \in \GF(q)^{nm\times nm}                         
        \]
        Conversely, if $\rho_q^{nm}(\eta)$ is block diagonal
        with block size $n$ (the diagonal blocks do not need to
        contain the same content), then $\eta \in \GF(q^n)$.
    \end{corollary}

    \begin{proof}
        The forward direction is a direct consequence of
        the block structure \eqref{block} and
        $\rho_{q^n}^m$ being an $\GF(q^n)$-algebra homomorphism.
        For the backward direction, consider
        \[
            \def\ddot{\raisebox{0.6ex}{$.$}.}
            \rho_q^{nm}(\eta) =
            \bma{Y_1 \\& \ddot \\&& Y_m}
            \in (\GF(q)^{n\times n})^{m\times m},
        \]
        where each $Y_i$ is an $n \times n$ matrix.
        By the block structure, each $Y_i$ is of the form
        $\rho_q^n(y_i)$ for some $y_i \in \GF(q^n)$,
        and hence $Y_i^{q^n} = Y_i$.
        This forces $\rho_q^{nm}(\eta)^{q^n} = \rho_q^{nm}(\eta)$,
        and hence $\eta^{q^n} = \eta$, leading to $\eta \in \GF(q^n)$.
    \end{proof}

\section{Representing the Frobenius Map}                \label{sec:frobenius}

    One nontrivial fact that was not mentioned in the preliminaries
    is that the Galois group is cyclically generated by the Frobenius map.

    \begin{fact} [Frobenius map]
        Fix a field extension $\FF_{q^n}/\FF_q$.
        The map $\varphi_q\colon \xi \mapsto \xi^q$
        is called the \emph{Frobenius map}.
        It is a field automorphism on $\FF_{q^n}$,
        has order $n$,
        fixes $\FF_q$ and nothing else,
        and generates the Galois group of $\FF_{q^n} / \FF_q$.
        See \cite[Theorem~2.21]{LNC09} for a proof.
    \end{fact}

    Since $\varphi_q$ has order $n$, it would be interesting
    to find a basis of $\FF_{q^n} / \FF_q$ such that
    the matrix representation $\varrho$ turns $\varphi_q$ into
    an action on matrices that is ``obviously'' cyclic.
    To demonstrate what we mean by that, consider the following
    representation of $\GF(8)$ different from \eqref{GF8}:
    \begin{equation*}
        \left\{
            \arraycolsep2.4pt
            \bsm{0 & 0 & 0 \\ 0 & 0 & 0 \\ 0 & 0 & 0},
            \bsm{1 & 1 & 0 \\ 1 & 1 & 1 \\ 0 & 1 & 0},
            \bsm{0 & 0 & 1 \\ 0 & 1 & 1 \\ 1 & 1 & 1},
            \bsm{0 & 1 & 0 \\ 1 & 0 & 1 \\ 0 & 1 & 1},
            \bsm{1 & 1 & 1 \\ 1 & 0 & 0 \\ 1 & 0 & 1},
            \bsm{0 & 1 & 1 \\ 1 & 1 & 0 \\ 1 & 0 & 0},
            \bsm{1 & 0 & 1 \\ 0 & 0 & 1 \\ 1 & 1 & 0},
            \bsm{1 & 0 & 0 \\ 0 & 1 & 0 \\ 0 & 0 & 1}
        \right\}
        \subset \GF(2)^{3\times3}
    \end{equation*}
    Call these matrices $0, B^1, \dotsc, B^7$.
    Now, apart from $B^i \cdot B^j = B^{(i+j)\%7}$, one also observes that
    \[ (B^i)^2 = P B^i P^\top, \qquad P \coloneqq \bsm{&&1 \\ 1&& \\ &1&}. \]
    That is to say, squaring the matrices is equivalent to shifting
    the rows and columns by the permutation $\psm{1 & 2 & 3 \\ 2 & 3 & 1}$.
    Since this is a cyclic permutation on three items,
    $(((B^i)^2)^2)^2 = P^3 B^i P^{\top3} = B^i$
    witnesses the fact that the Frobenius map has order $3$.
    Moreover, notice that the only matrices invariant under
    $P \bullet P^\top$ are the scalar matrices $0$ and $B^7$,
    which form a copy of $\GF(2)$ inside $\GF(8)$.
    This witnesses the fact that the Frobenius map
    fixes the base field and nothing else.

    The question is whether we can always
    represent the Frobenius map like this.

    \begin{fact} [normal basis]
        An element $\nu \in \GF(q^n)$ is said to be \emph{normal}
        over $\GF(q)$ if $\nu, \nu^q, \dotsc, \nu^{q^{n-1}}$
        form a basis of $\FF_{q^n} / \FF_q$.
        There always exists a normal element for any
        finite field extension \cite[Theorem~2.35]{LNC09}.
        In fact, there exists a primitive normal element
        whose trace is any prescribed nonzero value
        in the base field \cite{CoH99}.
        See \cite{KaR19,MKB25} for more recent results.
    \end{fact}

    \begin{corollary}
        Since $\varphi_q$ permutes a normal basis cyclically,
        the matrix representation $\varrho$ constructed
        from the normal basis satisfies the property that
        $\varrho(\xi)^q$ coincides with the result of
        cyclically permuting the rows and columns of $\varrho(\xi)$.
    \end{corollary}

    The affirmative answer only induces a deeper question:
    Can we find normal bases that also visualize the block structure
    of the subfield chains like earlier sections do?
    To this end, we propose the following.

    \begin{theorem} [Frobenius representation]          \label{thm:frobenius}
        Fix a prime power $q$.
        For all coprime degrees $n, m \ge 1$,
        there exists a matrix representation
        $\varrho_{q^n}^m\colon \GF(q^{nm}) \to \GF(q^n)^{m\times m}$
        together with a cyclic permutation matrix
        $P_{q^n}^m \in \{0, 1\}^{m\times m}$
        such that $\varrho_{q^n}^m$ is an injective
        $\GF(q^n)$-algebra homomorphism, hence a field embedding, and,
        for every
        $\xi \in \GF(q^{nm})$,
        \begin{equation}
            \varrho_{q^n}^m(\varphi_{q^n}(\xi)) =
            P_{q^n}^m \varrho_{q^n}^m(\xi) {P_{q^n}^m}^\top.   \label{cyclic}
        \end{equation}
        These maps can be made globally compatible in the sense that,
        for every triple of mutually coprime degrees $n, m, \ell \ge 1$,
        the composition
        \begin{equation}
            \def\.{\cdot}
            \def\r{\varrho_{q^n}^m}
            \def\M{\bsm{\r&\.\.&\r\\:&::&:\\\r&\.\.&\r}}
            \GF(q^{nm\ell}) \xrightarrow{\varrho_{q^{nm}}^\ell} 
            \GF(q^{nm})^{\ell\times\ell} \xrightarrow{\M}
            (\GF(q^n)^{m\times m})^{\ell\times\ell}
            \cong \GF(q^n)^{m\ell\times m\ell}                \label{coprime}
        \end{equation}
        coincides with $\varrho_{q^n}^{m\ell}\colon \GF(q^{nm\ell}) \to
        \GF(q^n)^{m\ell\times m\ell}$ up to row and column permutations.
    \end{theorem}

\subsection{An example of degree Twelve}

    Before the formal proof, let us use an example to illustrate the idea.
    Because the block-compatibility condition \eqref{coprime} only applies to
    coprime degrees, we do not have to consider factorizations like
    $12 = 2 \cdot 6$, but only the coprime ones like $12 = 4 \cdot 3$.
    That is to say, we only need to declare the basis
    for each extension whose degree is a prime power.

    For the degree-$3$ part, we find a normal element $\beta \in \GF(q^3)$
    and let $B$ be $\bma{\beta & \beta^q & \beta^{q^2}}$.
    For the degree-$4$ part, we find a normal element $\gamma \in \GF(q^4)$ 
    and let $C$ be $\bma{\gamma & \gamma^q & \gamma^{q^2} & \gamma^{q^3}}$.

    It remains to explain why $B \otimes C$
    is a normal basis of $\GF(q^{12}) / \GF(q)$.
    This is a direct consequence of the coprime condition:
    $\varphi_q$ acts on $B$ like the cyclic group $C_3$ of size $3$;
    $\varphi_q$ acts on $C$ like the cyclic group $C_4$ of size $4$.
    So the action of $\varphi_q$ on $B \otimes C$ is equivalent to
    the component-wise action of $C_3 \times C_4$ on $B \times C$.
    But $C_3 \times C_4$ is the cyclic group $C_{12}$ of size $12$,
    which is what we want.

\subsection{What happens when degrees are not coprime}

    From the previous example we see why the coprime conditions
    appear multiple times in Theorem~\ref{thm:frobenius}:
    It is because $C_n \times C_m$ is $C_{nm}$
    if and only if $n$ and $m$ are coprime.
    But this only means that our proof technique is not
    strong enough to handle the non-coprime case,
    not that our desired conclusion is bound to fail.

    Here, we demonstrate a ``counterexample''
    so authors of future works will know what to avoid:
    There are only two field homomorphisms from $\GF(4)$
    to $\GF(2)^{2\times2}$, and both images are \eqref{GF4}.
    While this does represent the Frobenius map by swapping
    the rows and columns, it strongly limits the choices
    of matrices for representing $\GF(16) / \GF(2)$.
    In fact, there are only $4^4 = 256$ matrices in
    $\eqref{GF4}^{2\times2} \subseteq \GF(2)^{4\times4}$
    and $6$ cyclic permutation matrices $P$.
    The only solutions to the equation $X^2 = P X P^\top$
    are the trivial ones: $X = 0 \cdot I_4$ and $X = I_4$.
    That is to say, no nontrivial representation of $\GF(16)$
    can represent the Frobenius map as a cyclic permutation
    while respecting the block structure.

\subsection{Proof of Theorem~\ref{thm:frobenius}}

    Now that we are convinced that Theorem~\ref{thm:frobenius}
    cannot be strengthened so easily, let us prove the current version.
    We follow the strategy suggested by the example above.
    First, for any prime power $n$,
    we find a normal element $\nu \in \GF(q^n)$ over $\GF(q)$.
    The degree-$n$ extension is achieved by the normal basis
    $N \coloneqq \bma{\nu & \nu^q & \cdots & \nu^{q^{n-1}}}$.

    Now, for any degree $m$, factorize $m$ into
    pairwise coprime prime powers $n_1 \dotsm n_l$.
    Find normal elements $\nu_1, \dotsc, \nu_l$ for those prime powers
    and construct the corresponding normal bases $N_1, \dotsc, N_l$.
    We then take the Kronecker product of these normal bases
    to be the basis for the degree-$m$ extension over $\GF(q)$.
    For \emph{relative} extensions such as
    $\GF(q^{n_1n_2n_3n_4n_5n_6}) / \GF(q^{n_4n_5n_6})$,
    we use the truncated product $N_1 \otimes N_2 \otimes N_3$ as the basis.

    The homomorphism condition is automatic because, after a basis is chosen,
    $\varrho$ is defined as the matrix representation of multiplication maps,
    as in Step 1 of Section~\ref{sec:main}.
    The normality of the basis has nothing to do
    with the validity of this argument.

    The cyclic permutation condition \eqref{cyclic} is satisfied
    when $n = 1$ because each $N_i$ is a normal basis,
    $\varphi_q$ acts cyclically on each $N_i$,
    and the component-wise action on $N_1 \times \cdots \times N_l$
    is a single cyclic action of order $m$.
    For $n > 1$, since $n$ is coprime to $m$,
    the map $\varphi_{q^n} = \varphi_q^n$ still acts cyclically on each $N_i$
    for each $n_i \mid m$, and so the same conclusion follows.

    Finally, the block-compatibility condition \eqref{coprime} is satisfied
    because our choice of basis possesses the Kronecker-product structure.
    Cf.\ step 2 of Section~\ref{sec:main}.
    This finishes the proof of Theorem~\ref{thm:frobenius}.

\section{Concluding Remarks}

    The two constructions above emphasize
    the same principle from different directions.
    Conway-compatible primitive elements organize inclusions among
    finite fields, while normal bases organize the Frobenius action.
    In both cases, the visible matrix patterns come from making
    the degree factorization visible at the level of bases.
    However, there does not seem to be a perfect way to add the
    Frobenius action to the picture of Theorem~\ref{thm:main}.
    It is therefore natural to ask what the next best possibilities are:
    how much of the block structure can be retained while also making
    the Frobenius action visible?

\bibliographystyle{alpha}
\bibliography{FiniteField-31}

\appendix

\section{Conway Polynomials}                               \label{sec:conway}

    To prove that Conway polynomials exist, we follow an induction:
    If all lower-degree polynomials satisfy the norm-compatibility
    conditions, then we find a generator $\kappa$ of $\GF(p^k)^*$ such that
    $\kappa^{(p^k-1)/(p^d-1)}$ generates $\GF(p^d)^*$ for every $d \mid k$.
    A formal proof can be found elsewhere so we only provide an example
    that demonstrates the idea better than a formal proof.

    Consider $k = 60$.
    We want to show that there exists $f_p^k$
    such that its root $\kappa$ satisfies:
    \begin{itemize}
        \item $\kappa^{(p^{60}-1)/(p^2-1)} = \omega_2$,
            the root of $f_p^2$ that was chosen to construct $\GF(p^2)$.
        \item $\kappa^{(p^{60}-1)/(p^3-1)} = \omega_3$,
            the root of $f_p^3$ that was chosen to construct $\GF(p^3)$.
        \item The same formulas for the remaining divisors
            $5$, $6$, $10$, $12$, $15$, $20$, and $30$.
    \end{itemize}
    Let $\omega_2, \omega_3, \dotsc, \omega_{59}$ be the roots of
    $f_p^2, f_p^3, \dotsc, f_p^{59}$ that have already been chosen
    by induction.
    The key idea here is that we only have to care about
    the maximal proper divisors of $k$, which are $30$, $20$, and $12$
    in this example, and are of the form $k/\text{prime}$ in general.

    Let $\lambda$ be a generator of $\GF(p^{60})^*$ and set
    \[ r^d_c \coloneqq \frac{p^d-1}{p^c-1} = \frac{[d]_p}{[c]_p} \]
    for any pair $c \mid d$.
    Since $\GF(p^{30})^*$, $\GF(p^{20})^*$, and $\GF(p^{12})^*$
    are the subgroups of $\GF(p^{60})^*$ of sizes
    $p^{30} - 1$, $p^{20} - 1$, and $p^{12} - 1$, respectively,
    $\lambda^{r^{60}_{30}}$, $\lambda^{r^{60}_{20}}$, and
    $\lambda^{r^{60}_{12}}$ generate them.
    And so the already-chosen roots $\omega_{30}$, $\omega_{20}$,
    and $\omega_{12}$ must be some powers of them.
    Let $s_{30}$, $s_{20}$, and $s_{12}$
    be the integers that witness these powers, i.e.,
    \[
        \omega_{30} = \lambda^{r^{60}_{30}s_{30}},\qquad
        \omega_{20} = \lambda^{r^{60}_{20}s_{20}},\qquad
        \omega_{12} = \lambda^{r^{60}_{12}s_{12}}.
    \]
    The induction hypothesis says that proper powers of
    $\omega_{30}$, $\omega_{20}$, and $\omega_{12}$
    should be compatible in smaller subfields;
    this leads to
    \begin{align*}
        \lambda^{r^{60}_{10}s_{30}}
        = \lambda^{r^{60}_{30}r^{30}_{10}s_{30}}
        = \omega_{30}^{r^{30}_{10}}
        = \omega_{10}
        = \omega_{20}^{r^{20}_{10}}
        = \lambda^{r^{60}_{20}r^{20}_{10}s_{20}}
        = \lambda^{r^{60}_{10}s_{20}}
        \in \GF(p^{10})^*,
        \\ \lambda^{r^{60}_{6}s_{30}}
        = \lambda^{r^{60}_{30}r^{30}_{6}s_{30}}
        = \omega_{30}^{r^{30}_{6}}
        = \omega_6
        = \omega_{12}^{r^{12}_{6}}
        = \lambda^{r^{60}_{12}r^{12}_{6}s_{12}}
        = \lambda^{r^{60}_{6}s_{12}}
        \in \GF(p^{6})^*,
        \\ \lambda^{r^{60}_{4}s_{20}}
        = \lambda^{r^{60}_{20}r^{20}_{4}s_{20}}
        = \omega_{20}^{r^{20}_{4}}
        = \omega_4
        = \omega_{12}^{r^{12}_{4}}
        = \lambda^{r^{60}_{12}r^{12}_{4}s_{12}}
        = \lambda^{r^{60}_{4}s_{12}}
        \in \GF(p^{4})^*.
    \end{align*}
    This forces the compatibility conditions on the $s$'s
    \[
        s_{30} \equiv s_{20} \pmod{p^{10} - 1}, \quad
        s_{30} \equiv s_{12} \pmod{p^6 - 1   }, \quad
        s_{20} \equiv s_{12} \pmod{p^4 - 1   }
    \]
    Hence the Chinese remainder theorem
    applies to the congruence equations
    \[
        t \equiv s_{30} \pmod{p^{30} - 1}, \qquad
        t \equiv s_{20} \pmod{p^{20} - 1}, \qquad
        t \equiv s_{12} \pmod{p^{12} - 1}.
    \]
    It remains to choose a solution $t$ and let $\kappa$ be $\lambda^t$.

    We now want to check why this $\kappa$ satisfies
    all the norm-compatibility conditions.
    This is straightforward for the maximal proper divisors.
    \begin{align*}
        \kappa^{r^{60}_{30}} = \lambda^{r^{60}_{30}t}
        = \lambda^{r^{60}_{30}s_{30}} = \omega_{30} \in \GF(p^{30})^*,
        \\ \kappa^{r^{60}_{20}} = \lambda^{r^{60}_{20}t}
        = \lambda^{r^{60}_{20}s_{20}} = \omega_{20} \in \GF(p^{20})^*,
        \\ \kappa^{r^{60}_{12}} = \lambda^{r^{60}_{12}t}
        = \lambda^{r^{60}_{12}s_{12}} = \omega_{12} \in \GF(p^{12})^*.
    \end{align*}
    For the compatibility conditions for smaller divisors,
    we use chain rules.
    For instance, we have
    \begin{align*}
        \kappa^{r^{60}_{15}} = \kappa^{r^{60}_{30}r^{30}_{15}}
        & = \omega_{30}^{r^{30}_{15}} = \omega_{15} \in \GF(p^{15})^*,
        \\ \kappa^{r^{60}_{10}} = \kappa^{r^{60}_{30}r^{30}_{10}}
        & = \omega_{30}^{r^{30}_{10}} = \omega_{10} \in \GF(p^{10})^*,
        \\ \kappa^{r^{60}_{6}} = \kappa^{r^{60}_{30}r^{30}_{6}}
        & = \omega_{30}^{r^{30}_{6}} = \omega_6 \in \GF(p^{6})^*.
    \end{align*}
    For the remaining divisors, apply more chain rules.

    The other thing we have to check is
    whether $\kappa$ generates $\GF(p^{60})^*$.
    Recall that $\lambda$ generates $\GF(p^{60})^*$, so $\kappa = \lambda^t$
    is primitive exactly when $t$ avoids all prime divisors of $p^{60} - 1$.
    This final step is handled by the following claim.

    \begin{claim}
        A solution $t$ of the congruence equations above always avoids all
        prime divisors of $p^{30} - 1$, $p^{20} - 1$, and $p^{12} - 1$.
        Thus some careful choice of $t$ avoids
        all prime divisors of $p^{60} - 1$.
    \end{claim}

    To prove the claim, note that $\omega_{30}$ is primitive,
    and so $s_{30}$ is coprime to $p^{30} - 1$.
    Similarly, $s_{20}$ and $s_{12}$
    are coprime to $p^{20} - 1$ and $p^{12} - 1$.
    This shows that any solution $t$ must avoid any prime divisor
    of the lcm of $p^{30} - 1$, $p^{20} - 1$, and $p^{12} - 1$.
    It remains to avoid the prime divisors of $p^{60} - 1$
    that do not appear in the lcm, but this is trivial because
    we can add an arbitrary multiple of the lcm to $t$.

\end{document}